\documentclass[12pt]{amsproc}
\usepackage{amssymb}

\usepackage[cp850]{inputenc}
\usepackage{mathrsfs}
\usepackage{amscd}
\usepackage[notcite,notref]{showkeys}
\usepackage[all]{xy}

\usepackage[colorlinks]{hyperref}

\theoremstyle{plain}
\newtheorem{theorem}{Theorem}[section]
\newtheorem{proposition}[theorem]{Proposition}
\newtheorem{lemma}[theorem]{Lemma}

\newtheorem{corollary}[theorem]{Corollary}

\theoremstyle{remark}
\newtheorem{definition}[theorem]{Definition}
\newtheorem{remark}[theorem]{Remark}

\def\PO{\operatorname{PO}}
\def\dom{\operatorname{dom}}
\def\cod{\operatorname{cod}}
\def\dens{\operatorname{dens}}

\def\R{\mathbb{R}}

\def\Gur{\mathbb{G}}
\def\Gp{\mathbb{G}_p}

\def\e{\varepsilon}

\newcommand{\fra}{Fra\"\i ss\'e}

\newcommand{\Nat}{\mathbb N}

\newcommand{\Pp}{\mathbb P_p}
\newcommand{\map}[3]{#1 \colon #2 \to #3}
\newcommand{\cmp}{\circ}
\newcommand{\Ha}{\mathbb H}
\newcommand{\fK}{\mathfrak H}

\newcommand{\norm}[1]{\|#1\|}
\newcommand{\K}{\mathbb K} 

\title{Quasi-Banach spaces of almost universal disposition}
\author[F. Cabello S\'anchez, J. Garbuli\'nska-W\c egrzyn, W. Kubi\'s]{F\'elix Cabello S\'anchez, Joanna Garbuli\'nska-W\c egrzyn, and Wies\l aw Kubi\'s\footnote{Research of FCS supported in part by MTM2010-20190-C02-01 and Junta de Extremadura GR10113.}
 \footnote{Research of WK supported by GA\v CR grant P 201/12/0290 and RVO: 67985840.}}
\address{Departamento de Matem\'{a}ticas, UEx, 06071-Badajoz, Spain}
\email{fcabello@unex.es}

\address{Institute of Mathematics, Jan Kochanowski University, \'Swietokrzyska 15, 25-406 Kielce, and
Institute of Mathematics, Faculty of Mathematics and Computer Science, Jagiellonian University, Lojasiewicza 6, 30-348 Krakow, Poland}
\email{jgarbulinska@ujk.edu.pl}

\address{Institute of Mathematics, Academy of Sciences of the Czech Republic, Zitn'25, 115 67 Praha 1, Czech Republic}
\email{kubis@math.cas.cz}

\begin{document}
\noindent{{\footnotesize April 29, 2014}\\[10pt]}

\begin{abstract}
{\bf We show that for each $p\in(0,1]$ there exists a separable $p$-Banach space $\mathbb G_p$ of almost universal disposition, that is, having the following extension property: for each $\e>0$ and each isometric embedding $g:X\to Y$, where $Y$ is a finite-dimensional $p$-Banach space and $X$ is a subspace of $\mathbb G_p$, there is an $\e$-isometry $f:Y\to \mathbb G_p$ such that $x=f(g(x))$ for all $x\in X$.

Such a space is unique, up to isometries, does contain an isometric copy of each separable $p$-Banach space and has the remarkable property of being ``locally injective'' amongst $p$-Banach spaces.

We also present a nonseparable generalization which is of universal disposition for separable spaces and ``separably injective''.
No separably injective $p$-Banach space was previously known for $p<1$.
}
\end{abstract}


\maketitle

\tableofcontents

\bigskip

\section{Introduction}

\subsection{Background}
In 1965, Gurari\u\i\ constructed a separable Banach space $\Gur$ of ``almost universal disposition for finite-dimensional spaces'', that is, having the following extension property: for every isometry $g: X \to Y$, where $Y$ is a finite-dimensional Banach space and $X$ is a subspace of $\Gur$, and every $\e>0$ there is an $\e$-isometry $f:Y\to \Gur$ such that $f(g(x))=x$ for all $x\in X$.

It is almost obvious that if $V$ is any other separable Banach space fitting in the definition, then there is a linear isomorphism $u:\Gur\to V$ with $\|u\|\cdot\|u^{-1}\|$ arbitrarily close to 1.

Gurari\u\i's creature spurred a considerable interest in Banach space theory and is still an object of intense research.  Amongst the main hits we find the following. Lusky proved in \cite{l-unique} that all separable Banach spaces of almost universal disposition are isometric; see \cite{ks} for a simpler proof. The space $\Gur$ is a Lindenstrauss space, that is, its dual space is isometric to an $L_1$-space. Moreover, every separable Lindenstrauss space is isometric to a subspace of $\Gur$ which is the range of a nonexpansive projection.
This was proved by Wojtaszczyk \cite{Wojt}, see also \cite[Proposition 8]{l-survey} where it is shown that one can arrange the projection so that its kernel is again isometric to $\Gur$. The Gurari\u\i\ space is complemented in no space of type $C(K)$ and it is isomorphic to the space of all continuous affine functions on the Poulsen simplex; see \cite[Corollary 2]{b-l} and \cite{l-survey}.

We refer the reader to the survey paper \cite{g-k} for more information and references reporting recent work and to \cite{group} for a related construction in group theory.

\subsection{The plan of the paper}
There is no clear intrinsic reason to restrict attention to Banach spaces in studying the extension of isometries.
In this paper  we push the notion of ``universal disposition'' and its relatives into the larger class of quasi-Banach spaces.

We shall construct, for each $p\in(0,1]$, a separable $p$-Banach space of almost universal disposition for finite-dimensional $p$-Banach spaces, which turns out to be unique, up to isometries, and that we will call $\Gp$.
Our main tools are the push-out construction and the notion of a direct limit, whose adaptations to the $p$-normed setting are presented in Sections~\ref{sec:po} and \ref{sec:direct}. The construction itself is carried out in Section~\ref{sec:main}.

In Section~\ref{sec:uni} we prove that any two separable $p$-Banach spaces of almost universal disposition for finite-dimensional $p$-Banach spaces are isometric. As a consequence, $\mathbb G_p$ contains an isometric copy of each separable $p$-Banach space, which improves a classical result by Kalton and provides a complete solution to an old problem in the isometric theory of quasi-Banach spaces.
Up to this point the paper is rather elementary and self-contained.

In Section~\ref{sec:w1} we present a nonseparable generalization.  We construct a $p$-Banach space whose density character is the continuum and which is of universal disposition for separable $p$-Banach spaces.  We also mention a result of Ben Yaacov and Henson, with a simpler argument provided by Richard Haydon, showing that it is impossible to reduce the size of the space in the preceding result.  We prove that these spaces contain isometric copies of all $p$-Banach spaces with density character $\aleph_1$ or less and that they are all isometric under the continuum hypothesis.

Section~\ref{sec:inj} studies the extension of operators with values in the spaces of (almost) universal disposition. Let us pause a moment for some definitions.
First, following a long standing tradition, a quasi-Banach space $E$ would be injective amongst $p$-Banach spaces if there is a constant $\lambda\geq 1$ such that for every $p$-Banach space $X$ and every subspace $Y$ of $X$ every operator $t:Y\to E$ extends to an operator $T:X\to E$ with $\|T\|\leq \lambda\|t\|$.
Also, we say that $E$ is separably injective amongst $p$-Banach spaces if the preceding condition holds for $X$ separable and we say that it is locally injective if it holds when $X$ is finite-dimensional.

After proving that there is no injective $p$-Banach space, apart from $0$, we show that $\mathbb G_p$ is ``locally injective'' (see Definition~\ref{DefVrsbtix}) and also that any space of universal disposition for separable $p$-Banach spaces is separably injective. No separably injective $p$-Banach space had been previously known for $p<1$.

In Section~\ref{sec:operas} we show the existence of a nonexpansive projection on $\Gp$ whose kernel is isometric to $\Gp$. Moreover, this projection is universal in the sense that the class of all its restrictions to closed subspaces contains (up to isometry) all possible nonexpansive operators from separable $p$-Banach spaces into $\Gp$.

Finally, the closing Section~\ref{closing} contains some miscellaneous remarks and questions which we found interesting.

\subsection{Quasi-Banach spaces}
We shall denote by $\K$ the field of scalars, which is fixed to be either the field of real or complex numbers.

A quasinorm on a $\K$-linear space $X$ is a function $\|\cdot\|:X\to \R_+$ satisfying the following conditions:
\begin{itemize}
\item If $\|x\|=0$, then $x=0$.
\item $\|\lambda x\|=|\lambda|\/\|x\|$ for every $\lambda\in\mathbb K$ and every $x\in X$.
\item There is a constant $C$ such that $\|x+y\|\leq C(\|x\|+\|y\|)$ for all $x,y\in X$.
\end{itemize}
A quasinorm gives rise to a linear topology on $X$, namely the least linear topology for which the unit ball $B=\{x\in X:\|x\|\leq 1\}$ is a
neighborhood of zero. This topology is locally bounded, that is, it has a bounded neighborhood of zero.
Actually, every locally bounded topology arises in this way. We refer the reader to \cite{kpr, r} for general information on locally bounded spaces.

A quasinormed space is a linear space $X$ equipped with a quasinorm. If $X$ is complete for the quasinorm topology we say that $X$ is a quasi-Banach space.

Let $p\in(0,1]$. A $p$-normed (respectively, $p$-Banach) space is a quasinormed (respectively, quasi-Banach) space whose quasinorm is a $p$-norm, that is, it satisfies the inequality $\|x+y\|^p\leq \|x\|^p+\|y\|^p$. The case $p=1$ corresponds to the popular class of Banach spaces. Observe that every $p$-norm is also a $q$-norm for each $q\leq p$.

It is an important result of Aoki and Rolewicz that every quasinorm is equivalent to a $p$-norm for some $p\in(0,1]$ in the sense that they induce the same topology; see \cite[Theorem 1.3]{kpr} or \cite[Theorem 3.2.1]{r}.

Let $X$ and $Y$ be quasinormed spaces. A linear map $f:X\to Y$ is a (bounded) operator if there is a constant $K$ such that $\|f(x)\|_Y\leq K\|x\|_X$ for all $x\in X$. The infimum of the constants $K$ for which the preceding inequality holds is denoted by $\|f\|$, the quasinorm of $f$.

 If, besides, one has $(1-\e)\|x\|_X\leq \|f(x)\|_Y\leq (1+\e)\|x\|_X$ for some $\e\in[0,1)$ independent of $x\in X$, then $f$ is called an $\e$-isometry.
 If $\|f(x)\|_Y=\|x\|_X$ for all $x\in X$, then $f$ is called an isometry. Isometries are not assumed to be surjective. However, we say that $X$ and $Y$ are isometric if there is a surjective isometry $f:X\to Y$.

Note that there is no quasi-Banach space containing, for every $\e>0$ and every $p\in(0,1]$, a subspace $\e$-isometric to the 2-dimensional space $\ell_p^2$, the space $\mathbb K^2$ with the $p$-norm $\|(s,t)\|_p=(|s|^p+|t|^p)^{1/p}$. So, strictly speaking, the title of the paper is a bit exaggerated.

\subsection{Push-outs}\label{sec:po}
This section is an adaptation of \cite[Section 2.1]{accgm} to the $p$-normed setting.

Let $(X_\gamma)_{\gamma\in\Gamma}$ be a family of $p$-Banach spaces, where $\Gamma$ is a set of indices. We set
$$\ell_p(\Gamma, X_\gamma) =\left\{(x_\gamma)\in \prod_{\gamma\in\Gamma} X_\gamma : \left(\sum_\gamma\|x_\gamma\|^p\right)^{1/p}<\infty\right\}
$$
with the obvious $p$-norm. If the family has two spaces only, say $X$ and $Y$, we just write $X\oplus_p Y$. It is important to realize that this construction represents the direct sum in the category of $p$-Banach spaces and nonexpansive operators in the obvious sense.

Let $u: X\to Y$ and $v:X\to Z$ be operators acting between $p$-normed spaces. The associated push-out diagram is
\begin{equation}\label{PO}
\begin{CD}
X@>u>> Y\\
@V v VV @VV v'V\\
Z@> u' >>\PO
\end{CD}
\end{equation}

Here, $\PO=\PO(u,v)$ is the quotient of the $p$-sum $Y\oplus_p Z$ by $S$, the closure of the subspace $\{(u(x),-v(x)): x\in X\}$. The map $u':Z\to \PO$ is the inclusion of $Z$ into $Y\oplus_p Z$, followed by the quotient map of  $Y\oplus_p Z$ onto $\PO = (Y\oplus_p Z)/S$. The operator $v'$ is defined analogously. The universal property behind this construction is that the preceding diagram is ``minimally commutative'', in the sense that if $v'':Y\to E$ and $u'':Z\to E$ are operators such that $u''\circ v=v''\circ u$, then there is a
unique operator $w:\PO\to E$ satisfying $u''=w\circ u'$ and $v''= w \circ v'$.
Clearly, $w((y,z)+S)=v''(y)+u''(z)$ and $\|w\|\leq \max(\|v''\|,\|u''\|)$.

As for the quasinorm of the operators appearing in (\ref{PO}) it is obvious that both $u'$ and $v'$ are nonexpansive. The proof of the following remark is left to the reader.

\begin{lemma}
Referring to Diagram~\ref{PO}, if $u$ is an isometry and $\|v\|\leq 1$, then $u'$ is an isometry.
\end{lemma}

\subsection{Direct limits}\label{sec:direct}
Let $(X_\alpha)$ be a family of $p$-Banach spaces indexed by a directed set $\Gamma$ whose preorder is denoted by $\leq$. Suppose that, for each $\alpha,\beta\in \Gamma$ with $\alpha\leq \beta$ we have an isometry $f_\alpha^\beta: X_\alpha\to X_\beta$ in such a way that $f_\alpha^\alpha$ is the identity on $X_\alpha$ for every $\alpha\in\Gamma$ and $f_\beta^\gamma\circ f_\alpha^\beta=f_\alpha^\gamma$ provided $\alpha\leq \beta\leq \gamma$. Then $(X_\alpha, f_\alpha^\beta)$ is said to be a directed system of $p$-Banach spaces.

The direct limit of the system is constructed as follows.
First we take the disjoint union $\bigsqcup_\alpha X_\alpha$ and we define an equivalence relation $\sim$ by identifying $x_\alpha\in X_\alpha$ and $x_\beta\in X_\beta$ if there is $\gamma\in\Gamma$ such that $f_\alpha^\gamma(x_\alpha)=f_\beta^\gamma(x_\beta)$.

Then we may use the natural inclusion maps $\imath_\gamma: X_\gamma\to \bigsqcup_\alpha X_\alpha$ to transfer the linear structure and the $p$-norm from the spaces $X_\alpha$ to $\bigsqcup_\alpha X_\alpha/\!\sim$ thus obtaining a $p$-normed space whose completion is called the direct limit of the system and is denoted by $\underrightarrow{\lim} X_\gamma$. The universal property behind this construction is the following: if we are given a system of nonexpansive operators $u_\gamma:X_\gamma\to Y$, where $Y$ is a $p$-Banach space, which are compatible with the $f_\alpha^\beta$ in the sense that $u_\alpha=u_\beta\circ f_\alpha^\beta$ for $\alpha\leq \beta$, then there is a unique nonexpansive operator $u: \underrightarrow{\lim}\: X_\gamma \to Y$ such that $u\circ \imath_\alpha=u_\alpha$ for every $\alpha\in \Gamma$. That operator is often called the direct limit of the family $(u_\alpha)$.

\section{Construction of $p$-Banach spaces of almost universal disposition}\label{sec:main}

Let $\mathscr C$ be a class of quasi-Banach spaces.
Following \cite[Definition 2]{g}, let us say that a quasi-Banach space $U$ is of  almost universal disposition for the class $\mathscr C$ if, for  every $\e>0$ and for every isometry $g:X\to Y$, where $Y$ belongs to $\mathscr C$ and $X$ is a subspace of $U$, there is an $\e$-isometry $f:Y\to U$ such that $f(g(x))=x$ for all $x\in X$.

Here is the  main result of the paper.

\begin{theorem}\label{main}
For every $p\in(0,1]$ there exists a unique, up to isometries, separable $p$-Banach space of almost universal disposition for finite-dimensional $p$-Banach spaces. This space contains an isometric copy of every separable $p$-Banach space.
\end{theorem}

From now on we fix $p\in(0,1]$ once and for all.
We remark that everything in this paper is well-known for $p=1$.
However, the spaces we shall construct have rather unexpected properties when $p<1$ and shed some light on a widely ignored paper by Kalton \cite{k78}, where one can find a forerunner of our construction; see Proposition~\ref{inj} below.

Concerning the last statement of Theorem~\ref{main}, it is perhaps worth noticing that, while it is well-known that the separable Banach space $C[0,1]$ (as well as $\mathbb G$) contains an isometric copy of every separable Banach space, there is no available proof of the corresponding fact for $p$-Banach spaces for $p<1$. In \cite[Theorem 4.1(a)]{k77} it is stated without proof that for $0<p<1$ there exists a separable $p$-Banach space which is ``universal'' for the class of all separable $p$-Banach spaces. This result also appears in \cite[Theorem 3.2.8]{r} but, as far as we can understand, the rather involved  proof
only gives ``universality with respect to $\e$-isometries''.


Before embarking into the proof of Theorem~\ref{main}, let us record the following remark.

\begin{lemma}\label{relax}
Let $U$ be a $p$-Banach space. We assume that for every $\e>0$ and every isometry $g:X\to Y$, where $Y$ is a finite-dimensional $p$-Banach space and $X$ is a subspace of $U$, there is an $\e$-isometry $f:Y\to U$ such that $\|f(g(x))-x\|\leq \e\|x\|$ for all $x\in X$.

Then $U$ is of almost universal disposition for finite-dimensional $p$-Banach spaces.
\end{lemma}

\begin{proof}
This obviously follows from the fact that if  $B$ is a basis of $Y$, then for every $\e>0$ there is $\delta$ (depending on $\e$ and $B$) such that if $t:Y\to U$ is linear map with $\|t(b)\|\leq \delta$ for every $b\in B$, then $\|t\|\leq \e$.
\end{proof}

The following result, which should be compared to \cite[Lemma~4.2]{k78} and the construction in \cite[Section 3]{accgm}, is the key step in our construction. It is assumed that the families  $\frak J$ and  $\frak L$ are actually sets.

\begin{lemma}\label{po}
Let $E$ be a $p$-Banach space, $\frak J$ be a family of isometric embeddings between $p$-Banach spaces and $\frak L$ a family of non-expansive operators (i.e. $\|f\| \leq 1$ for every $f \in \frak L$) from $p$-Banach spaces into $E$. Then there is a $p$-Banach space $E'$ and an isometry $\imath: E\to E'$ having the following property: if $u: A\to B$ is in $\frak J$ and $f:A\to E$ is in $\frak L$, then there is $f':B\to E'$ such that $f'\circ u=\imath\circ f$, with $\|f'\|=\|f\|$. Moreover, if $f$ is an $\e$-isometry, then $f'$ is an $\e$-isometry too.
\end{lemma}

\begin{proof}
If $f:X\to Y$ is an operator, then we put $\dom(f) := X$ and $\cod(f) := Y$. Note that $\cod(f)$ may be larger than the range of $f$.
Set $\Gamma=\{(u,t)\in \mathfrak J\times \frak L: \dom(u)=\dom(t)\}$.
 We consider the spaces of $p$-summable families $\ell_p(\Gamma, \dom(u))$ and
$\ell_p(\Gamma, \cod(u))$. We have an isometry
$
\oplus\frak J: \ell_p(\Gamma, \dom(u))\to \ell_p(\Gamma, \cod(u))
$
given by $\oplus\frak J((x_{(u,t)})_{(u,t)\in \Gamma})= (u(x_{(u,t)}))_{(u,t)\in \Gamma}$.
In a similar vein, we can define a nonexpansive operator $\sum\frak L:  \ell_p(\Gamma, \dom(u))\to E$ by letting $\sum\frak L((x_{(u,t)})_{(u,t)\in \Gamma})=
\sum_{(u,t)\in \Gamma} t(x_{(u,t)})$. The notation is slightly imprecise because both operators depend on $\Gamma$.

Now we can consider the push-out diagram
\begin{equation}\label{lp}
\begin{CD}
 \ell_p(\Gamma, \dom(u))@> \oplus\frak J >> \ell_p(\Gamma, \cod(u))\\
@V \sum\frak L VV   @V (\sum\frak L)' VV\\
E @> (\oplus\frak J)' >>\PO
\end{CD}
\end{equation}
Let us see that the lower arrow does the trick so that we may take $E'=\PO$ and $\imath=(\oplus\frak J)'$. We already know that $(\oplus\frak J)'$ is an isometry and also that $(\sum\frak L)'$ is nonexpansive.

Fix $(v,s)$ in $\Gamma$. Put $X=\dom(v)=\dom(s)$ and $Y=\cod(v)$. Let $s'$ be the inclusion $\imath_{(v,s)}$ of $Y$ into the $(v,s)$-th coordinate of $\ell_p(\Gamma, \cod(u))$ followed by $(\sum\frak L)'$. As Diagram (\ref{lp}) is commutative, it is clear that $s'\circ v=(\oplus\frak J)'\circ s$ and also that $s'$ is nonexpansive.

Now suppose $s$ is an $\e$-isometry, that is, $(1-\e)\|x\|_X\leq \|s(x)\|_Y\leq \|x\|_X$ (recall that $s$ is nonexpansive). For $y\in Y$ one has
$$
\|s'(y)\|_{\PO}=\|(\imath_{(v,s)}(y),0)+S\|_{\ell_p(\Gamma, \cod(u))\oplus_p E},
$$
where $S=\{((\oplus\frak J)((x_{(u,t)})), -(\sum\frak L)((x_{(u,t)}))): (x_{(u,t)}))_{(u,t)\in \Gamma}\in\ell_p(\Gamma, \dom(u) ) \}$.

Clearly,
$$
\left\|\imath_{(v,s)}(y)-(u(x_{(u,t)}))_{(u,t)}\right\|^p_{\ell_p(\Gamma, \cod(u) )}+ \left\|\sum_{(u,t)\in\Gamma} t(x_{u,t})\right\|^p_E\geq \|y-v(x)\|_Y^p+\|s(x)\|_E^p,
$$
where $x=x_{(v,s)}$. Now, if $\|x\|_X\geq \|y\|_Y$ one has
$$
 \|y-v(x)\|_Y^p+\|s(x)\|_E^p\geq \|s(x)\|_E^p\geq (1-\e)^p\|x\|_X^p\geq (1-\e)^p\|y\|_Y^p.
$$
If  $\|x\|_X\leq \|y\|_Y$, then
\begin{align*}
 \|y-v(x)\|_Y^p+\|s(x)\|_E^p&\geq \|y\|_Y^p-\|v(x)\|_Y^p+(1-\e)^p\|x\|_X^p\\
&\geq \|y\|_Y^p-(1-(1-\e)^p)\|x\|_X^p\\
&\geq  (1-\e)^p\|y\|_Y^p.
\end{align*}
Thus, $
\|s'(y)\|_{\PO}\geq (1-\e)\|y\|_Y$ and $s'$ is a nonexpansive $\e$-isometry.
\end{proof}

\begin{lemma}\label{contains}
Every separable $p$-Banach space is isometric to a subspace of a separable $p$-Banach space of almost universal disposition.
\end{lemma}

\begin{proof}
Let $\frak F$ be a countable family of isometries between finite-dimensional $p$-normed spaces having the following density property: for every isometry of finite-dimensional $p$-normed spaces $g:A\to B$ and every $\e\in(0,1)$ there is $f\in\frak F$ and surjective $\e$-isometries $u:A\to \dom(f)$ and $v:B\to\cod(f)$ making commutative the square
$$
\begin{CD}
A@>g>> B\\
@Vu VV @V v VV\\
\dom(f) @>> f >\cod(f)
\end{CD}
$$
Let $S$ be a separable $p$-Banach space.
We shall construct inductively a chain of separable $p$-Banach spaces based on the nonnegative integers
$$
\begin{CD}
G_0@>\imath_1>>\dots@>>> G_{n-1}@>\imath_n>> G_n@>\imath_{n+1}>> G_{n+1}@>>>\dots
\end{CD}
$$
as follows. We put $G_0=S$ and, assuming that $G_k$ and $\imath_k$ have been constructed for $k\leq n$,
we take a countable set of operators $\mathfrak L_n$ such that for every $\e\in(0,1)$, every $f\in\frak F$ and every $\e$-isometry $u:\dom(f)\to G_n$, there is $v\in \frak L_n$ satisfying $\|u-v\|<\e$.

Then, we apply Lemma~\ref{po} with $E=G_n, \mathfrak J=\mathfrak F, \mathfrak L=\mathfrak L_n$ and we set $G_{n+1}=E'$ and $\imath_{n+1}=\imath$.

Finally, we consider the direct limit
$$
\mathbb G_p(S)= \underrightarrow{\lim}\: G_n
$$
and we prove that it satisfies the hypothesis of Lemma~\ref{relax}. 

So suppose we are given an isometry $g:X\to Y$, where $Y$ is a finite-dimensional $p$-Banach space and $X$ is subspace of $\mathbb G_p(S)$ and $\e>0$.
We shall prove that there is an $\e$-isometry $f:Y\to \mathbb G_p(S)$ such that $\|f(g(x))-x\|\leq \e\|x\|$ for all $x\in X$.

Let us fix $\delta>0$. The precise value of $\delta$ required here will be announced later.

First, there is an integer $n$ and a linear map $w:X\to G_n$ such that $\|w(x)-x\|\leq\delta\|x\|$. Moreover, we may take $h\in \mathfrak F$ and $\delta$-isometries $u$ and $v$ making the following diagram commutative:
$$
\begin{CD}
\dom(h) @> h >>\cod(h)\\
@V u VV @VV v V\\
X @> g >>Y
\end{CD}
$$
In fact we can clearly assume that $t=w\circ u$ is in $\frak L_n$ and also that it is a $\delta$-isometry.

Let $t':\cod(h)\to G_{n+1}$ be a $\delta$-isometry extending $t$ and set $f=t'\circ v^{-1}$.
Obviously $\|f(g(x))-x\|=\|w(x)-x\|\leq\delta\|x\|$ for all $x\in X$. Moreover,
$$
(1-\delta)^2\|y\|\leq \|f(y)\|\leq (1+\delta)^2\|y\|\quad\quad(y\in Y)
$$
and therefore taking $\delta=\sqrt{1+\e}-1$ suffices.
\end{proof}

\begin{remark}\label{RmRationsl}
In order to obtain a countable family of isometries having the property required in the proof of Lemma~\ref{contains} one may proceed as follows.
Let us say that a vector in $\K^n$ is ``rational'' if its components are all rational -- here, a complex number is ``rational'' if both its real and imaginary parts are rational numbers. Let $x_1,\cdots, x_k$ be rational vectors spanning $\mathbb K^n$ and put
$$
|x|=\inf\left\{ \left(\sum_{i=1}^k |\lambda_i|^p\right)^{1/p} \colon x= \sum_{i=1}^k \lambda_i x_i	\right\}.
$$
Then $|\cdot|$ is a $p$-norm on $\mathbb K^n$ and  we say that $(\mathbb K^n, |\cdot|)$ is a \emph{rational $p$-normed space}.
Consider the family of those isometries $f$ whose codomain is a rational $p$-normed space $(\mathbb K^n,|\cdot|)$, its domain is $\mathbb K^m$ for some $m\leq n$, equipped with a (not necessarily rational) $p$-norm and having the form $f(x_1,\ldots, x_m)= (x_1,\ldots, x_m, 0,\ldots, 0)$. Then an obvious compactness argument shows that these are ``dense amongst all isometries''.
\end{remark}

\section{Uniqueness}\label{sec:uni}

The following result is the first step towards the proof of uniqueness in Theorem~\ref{main}. It is the $p$-convex analogue of \cite[Lemma 2.1]{ks}.
As the reader can imagine, the proof has to be different here since one needs to avoid the use of linear functionals to work with $p$-normed spaces.

\begin{lemma}\label{key}
Let $X$ and $Y$ be $p$-normed spaces and $f:X\to Y$ an $\e$-isometry, with $\e\in(0,1)$. Let $i:X\to X\oplus Y$ and $j:Y\to X\oplus Y$ be the canonical inclusions. Then there is a $p$-norm on $X\oplus Y$ such that $\|j\circ f - i\|\leq\e$ and both $i$ and $j$ are isometries.
\end{lemma}

\begin{proof}Put
\begin{equation}\label{put}
\|(x,y)\|^p=\inf\left\{\|x_0\|_X^p+\|y_1\|_Y^p+\e^p\|x_2\|_X^p: (x,y)=(x_0,0)+(0,y_1)+(x_2,-f(x_2))		\right\}.
\end{equation}
It is easily seen that this formula defines a $p$-norm on $X\oplus Y$. Let us check that $\|(x,0)\|=\|x\|_X$ for all $x\in X$. The inequality  $\|(x,0)\|\leq \|x\|_X$ is obvious. As for the converse,  suppose $x=x_0+x_2$ and $y_1=f(x_2)$. Then
\begin{align*}
\|x_0\|_X^p+\|y_1\|_Y^p+\e^p\|x_2\|_X^p&= \|x_0\|_X^p+\|f(x_2)\|_Y^p+\e^p\|x_2\|_X^p\\
&\geq  \|x_0\|_X^p+(1-\e)^p\|x_2\|_X^p+\e^p\|x_2\|_X^p\\
&=  \|x_0\|_X^p+\|(1-\e)x_2\|_X^p+\|\e x_2\|_X^p\\
&\geq \|x\|_X^p,
\end{align*}
as required.

Next we prove that $\|(0,y)\|=\|y\|_Y$ for every $y\in Y$. That  $\|(0,y)\|\leq\|y\|_Y$ is again obvious. To prove the reversed inequality assume $x_0+x_2=0$ and $y=y_1-f(x_2)$. As $t\to t^p$ is subadditive on $\R_+$ for $p\in(0,1]$, we have
\begin{align*}
\|x_0\|_X^p+\|y_1\|_Y^p+\e^p\|x_2\|_X^p&= \|x_2\|_X^p+\|y_1\|_Y^p+\e^p\|x_2\|_X^p\\
&=  \|y_1\|_Y^p+(1+\e^p)\|x_2\|_X^p\\
&\geq  \|y_1\|_Y^p+(1+\e)^p\|x_2\|_X^p\\
&\geq \|y_1\|_Y^p+\|f(x_2)\|_Y^p\\
&\geq \|y\|_Y^p.
\end{align*}
To end, let us estimate $\|j\circ f- i\|$. We have
$$
\|j\circ f- i\|=\sup_{\|x\|\leq 1}\|j(f(x))-i(x)\|= \sup_{\|x\|\leq 1}\|(-x,f(x))\|\leq \e
$$
and we are done.
\end{proof}

From now on, $X\oplus_f^\e Y$ will denote the sum space $X\oplus Y$ furnished with the quasinorm defined by (\ref{put}). The fact that the quasinorm depends, not only on $f$ and $\e$, but also on $p$ will cause no confusion.

A linear operator $f:X\to Y$ is called a strict $\e$-isometry if for every nonzero $x\in X$,
 $$(1-\e)\|x\|_X< \|f(x)\|_Y< (1+\e)\|x\|_X,$$ where $\e\in(0,1)$.
Note that when $X$ is finite-dimensional, every strict $\e$-isometry is an $\eta$-isometry for some $\eta < \e$.

\begin{lemma}\label{helpful}
Let $U$ be a $p$-Banach space of almost universal disposition for finite-dimensional $p$-Banach spaces and let $f:X\to Y$ be a strict $\e$-isometry, where $Y$ is a finite-dimensional $p$-Banach space, $X$ is a subspace of $U$ and $\e\in(0,1)$. Then for each $\delta>0$ there exists a $\delta$-isometry $g:Y\to U$ such that $\| g(f(x))-x	\|< \e\|x\|$ for every nonzero $x\in X$.
\end{lemma}

\begin{proof}
Choose $0<\eta<\e$ such that $f$ is an $\eta$-isometry. Shrinking $\delta$ if necessary, we may assume that $\delta^p+(1+\delta)^p\eta^p<\e^p$. Set $Z= X\oplus_f^\eta Y$ and let $i:X\to Z$ and $j:Y\to Z$ denote the canonical inclusions, so that $\|j\circ f-i\|\leq \eta$. Let $h:Z\to U$ be a $\delta$-isometry such that $\|h(i(x))-x\|\leq\delta\|x\|$ for $x\in X$. Then $g=h\circ j$ is a $\delta$-isometry from $Y$ into $U$ and we have
\begin{align*}
\|x-g(f(x))\|^p &\leq \|x-h(i(x))\|^p+\|h(i(x))-h(j(f(x))\|^p\\
&\leq \delta^p\|x\|^p+(1+\delta)^p\|i(x)-j(f(x))\|_Z^p\\
&\leq (\delta^p+(1+\delta)^p\eta^p)\|x\|^p < \e^p \|x\|^p,
\end{align*}
as required.
\end{proof}

We are now ready for the proof of the uniqueness. Note that the following result, together with Lemma~\ref{contains}, completes the proof of
Theorem~\ref{main}.

\begin{theorem}\label{completes}
Let $U$ and $V$ be separable $p$-Banach spaces of  almost universal disposition for finite-dimensional $p$-Banach spaces. Let $f:X\to V$ be a strict $\e$-isometry, where $X$ is a finite-dimensional subspace of $U$ and $\e\in(0,1)$. Then there exists a bijective
isometry $h:U\to V$ such that $\|h(x)-f(x)\|_V\leq \e\|x\|_U$ for every
$x\in X$. In particular, $U$ and $V$ are isometrically isomorphic.
\end{theorem}

\begin{proof}
Fix $0<\eta <\e $ such that $f$ is an $\eta $-isometry and then choose $0<\lambda <1$ such that
\begin{equation}
\eta^p \frac{1+3\lambda^p}{1-\lambda^p }<\e^p.
\tag{$\star $}\label{star}
\end{equation}
Let $\e_n=\lambda^{n} \eta$. We define inductively sequences of linear operators $(f_n), (g_n)$ and finite-dimensional subspaces $(X_n)$, $(Y_n)$ of $U$ and $V$, respectively, so that the following conditions are satisfied:
\begin{enumerate}
    \item[(0)] $X_0 = X$, $Y_0 = f[X]$, and $f_0 = f$;
    \item[(1)] $f_n:X_n\to Y_n$ is an $\e_n$-isometry;
    \item[(2)] $g_n:Y_n\to X_{n+1}$ is an $\e_{n+1}$-isometry;
    \item[(3)] $\|g_n f_n(x) - x\| < \e_n  \| x\|$ for $x\in X_n$;
    \item[(4)] $\|f_{n+1}g_n(y) - y\| < \e_{n+1} \| y\|$ for $y\in Y_n$;
    \item[(5)] $X_n\subset X_{n+1}$, $Y_n\subset Y_{n+1}$, $\bigcup _n X_n$ and $\bigcup _n Y_n$ are dense in $U$ and $V$, respectively.
\end{enumerate}
We use condition (0) to start the inductive construction. Suppose that $f_i$, $X_i$, $Y_i$, for $i\leq n$, and $g_i$ for $i<n$, have been constructed. We easily find $g_n$, $X_{n+1}$, $f_{n+1}$ and $Y_{n+1}$ using Lemma~\ref{helpful}.

 To guarantee that Condition (5) holds, we may start by taking sequences $(x_n)$ and $(y_n)$ dense in $U$ and $V$, respectively and then we require first that $X_{n+1}$ contains both $x_n$ and $g_n[Y_n]$ and then that $Y_{n+1}$ contains both $y_n$ and $f_{n+1}[X_{n+1}]$. 

After that, fix $n \in \omega$ and $x \in X_n$ with $\|x\|=1$. Using (4), we
get
$$\| f_{n+1} g_n f_n(x) - f_n(x) \|^p < \e_{n+1}^p \cdot \|f_n(x)\|^p \leq  \e_{n+1}^p\cdot (1 + \e_{n})^p=(\lambda^{n+1}\eta)^p \cdot(1+\lambda^{n}\eta)^p.$$
Using (3), we get
$$\| f_{n+1} g_n f_n(x) - f_{n+1}(x) \|^p \leq  \|f_{n+1}\|^p \cdot \| g_n f_n(x) - x\|^p < (1 + \e_{n+1})^p\cdot \e_{n}^p=(\lambda^{p}\eta)^p \cdot(1+\lambda^{n+1}\eta)^p.$$
These inequalities give
\begin{align*}
\| f_n(x) - f_{n+1}(x) \|^p & < (\lambda^{n}\eta  + \lambda^{n}\eta  \lambda^{n+1}\eta)^p  + (\lambda^{n}\eta  \lambda^{n+1}\eta + \lambda^{n+1}\eta)^p \\ &< \eta^p (\lambda^{np}+2\lambda^{(n+1)p}+\lambda^{(n+1)p}) =\eta^p (\lambda^{np} + 3\lambda^{(n+1)p}).
\tag{$\star \star $}\label{twostar}
\end{align*}
Now it is clear that $(f_n(x))_{n\in \omega}$ is a  Cauchy sequence.
Given $x\in \bigcup_{n \in \omega} X_n$, define $h(x) = \lim_{n \geq
m}f_n(x)$, where $m$ is such that $x\in X_m$. Then $h$ is an
$\e_n$-isometry for every $n \in \omega$, hence it is an isometry.
Consequently, it  extends to an isometry on $h: U\to V$ that we do not relabel. Furthermore, (\ref{star})
and (\ref{twostar}) give
\begin{align*}
\| f(x) - h(x)\|^p  \leq \sum_{n=0}^\infty \eta^p (\lambda^{np} + 3\lambda^{(n+1)p})=\eta^p \frac{1+3\lambda^p}{1-\lambda^p} <\e^p
\end{align*}
It remains to see that $h$ is a bijection. To this end, we check
as before that $(g_n(y))_{n \geq  m}$ is a Cauchy sequence for
every $y\in Y_{m}$. Once this is done, we obtain an isometry
$g:V\to U$. Conditions (3) and (4) tell us that
$g\circ  h$ is the identity on $U$ and that $h \circ  g$ is the identity on $V$. This
completes the proof.
\end{proof}

\section{Nonseparable generalizations}\label{sec:w1}

As the reader may expect, we say that a quasi-Banach space $U$ is of universal disposition for a given class of quasi-Banach spaces $\mathscr C$ if, whenever $g:X\to Y$ is an isometry, where $Y$ belongs to $\mathscr C$ and $X$ is a subspace of $U$, then there is an isometry $f:Y\to U$ such that $f(g(x))=x$ for all $x\in X$.

Using $\mathbb G_p$ as an isometrically universal separable $p$-Banach space and iterating Lemma~\ref{po} until the first uncountable ordinal $\omega_1$ we now proceed as in \cite[Proposition 3.1(a)]{accgm} to prove the following.

\begin{theorem}\label{w1}
There is a $p$-Banach space of universal disposition for separable $p$-Banach spaces and whose density character is the continuum.
\end{theorem}

\begin{proof} Let $\omega_1$ be the first uncountable ordinal. We may regard $\omega_1$ as the set of all countable ordinals equipped with the obvious order; see \cite{cie} for details. We are going to define a transfinite sequence of $p$-Banach spaces $(G_p^\alpha, f_\alpha^\beta)$ indexed by $\omega_1$ having the following properties:
\begin{itemize}
\item[(1)] For each $\alpha\in\omega_1$ the density character of $G_p^\alpha$ is at most the continuum.
\item[(2)] If $\beta=\alpha+1$ and $g:X\to Y$ is an isometry, where $Y$ is a separable $p$-Banach space and $X$ is a subspace of $G_p^\alpha$, then there is an isometry $f:Y\to G_p^{\beta}$ such that $f(g(x))=f_\alpha^{\beta}(x)$ for all $x\in X$.
\end{itemize}

We proceed by transfinite induction on $\alpha\in\omega_1$. Let us fix an arbitrary $p$-Banach space $C$ with density $2^{\aleph_0}$.  Then, we take $G_p^0=C$ to start.

 The inductive step is as follows. We fix $\gamma\in\omega_1$ and we assume that the directed system $(G_p^\alpha,f_\alpha^\beta)$ has been constructed for $\alpha,\beta<\gamma$ in such a way that (1)  and (2) hold for $\alpha,\beta<\gamma$.

 We want to see that we can continue the system in such a way that (1) and (2) now hold for $\alpha,\beta<\gamma+1$. We shall distinguish two cases.

 First, assume $\gamma$ is a limit ordinal. Then we take $G_p^\gamma=\underrightarrow{\lim}_{\alpha<\gamma} G_p^\alpha$ and $f_\alpha^\gamma=\imath_\alpha$. It is clear that $\dens(G_p^\gamma)\leq 2^{\aleph_0}$ and there is nothing else to prove since $\gamma$ cannot arise as $\alpha+1$ for $\alpha<\gamma$.

 Now, suppose $\gamma$ is a successor ordinal, say $\gamma=\delta+1$.
To construct $G_p^{\delta+1}$ we consider the set of all isometric embeddings between subspaces of $\mathbb G_p$ and we call it $\mathfrak J$ and the set $\mathfrak L$ of all $G_p^\delta$-valued isometries whose domain is a subpace of $\mathbb G_p$ -- recall that $G_p^\delta$ is already defined by the induction hypothesis.  Now, we let $E= G_p^\delta$ and we apply Lemma~\ref{po} with $\e=0$ to get the push-out space $G_p^{\delta+1}=E'$ and $f_\delta^{\delta+1}=\imath$. Observe that $G_p^{\delta+1}$ has density character at most the continuum since it is a quotient of the direct sum of $G_p^\delta$ and $\ell_p(\Gamma, \cod(u))$, where $\Gamma$ is a subset of $\mathfrak J\times \mathfrak L$, with   $|\mathfrak J|,|\mathfrak L|\leq\frak c$ and $\cod(u)$ separable for every $u$.

Now, for $\alpha<\delta$ we put $f_\alpha^{\delta+1}= f_\delta^{\delta+1}\circ f_\alpha^{\delta}$ and the Principle of Transfinite Induction goes at work.

The remainder of the proof is rather easy. We define $U$ as the direct limit of the system $(G_p^\alpha)_\alpha$ and we consider the natural isometries $\imath_\alpha: G_p^\alpha\to U$, so that
$$
U=\bigcup_{\alpha\in\omega_1} \imath_\alpha [G_p^\alpha].
$$
Observe that it is not necessary to take closures here.
Obviously, the density character of $U$ is at most the continuum. 

Suppose $g:X\to Y$ is an isometry, where $Y$ is a separable $p$-Banach space and $X$ a subspace of $
U$. Then there is $\alpha\in\omega_1$ so that $X\subset\imath_\alpha[G_p^\alpha]$. It is straightforward from (2) that there is an isometry $f:Y\to G_p^{\alpha+1}$ such that $\imath_{\alpha+1}(f(g(x)))=x$ for every $x\in X$.
\end{proof}

The following result, due to Ben Yaacov and Henson~\cite{BYH},
shows that it is impossible to reduce the size of the space in Theorem~\ref{w1}.
We include a nice, straightforward proof found by
Richard Haydon.

 Formally, Ben Yaacov and Henson stated and proved the result for $p=1$, however Haydon's argument gives exactly the same for any $p \in (0,1]$.
Namely, a $p$-Banach space of universal disposition for the class of $p$-Banach spaces of dimension three must already have density $2^{\aleph_0}$. This was asked in \cite[Problem 2]{accgm} for $p=1$.

\begin{proposition}[Ben Yaacov and Henson]\label{PBenYakHay}
Let $H$ be the 2-dimensional Hilbert space and suppose $X$ is a $p$-Banach space containing $H$ and having the following property:
\begin{enumerate}
\item[$(\checkmark)$] Given an isometric embedding $\map i H F$, where $F$ is a $3$-dimensional $p$-Banach space, there exists an isometric embedding $\map j F X$ such that $j \cmp i$ is the inclusion $H \subset X$.
\end{enumerate}
Then the density of $X$ is at least the continuum.
\end{proposition}

\begin{proof} (Haydon)
Let $S$ be the positive part of the unit sphere of $H$.
Given $\phi \in S$, we define a $p$-norm on $H \oplus \K$ (recall that $\K$ is the scalar field) by the formula
$$\norm{(x,\lambda)}^p_\phi = \max \Bigl\{ \norm x_2^p, |\lambda|^p + |(x | \phi)|^p \Bigr\},$$
where $(\cdot|\cdot)$ denotes the usual scalar product on $H$.
Note that $\norm{(0,1)}_\phi = 1$ and $\norm{\cdot}_\phi$ extends the Euclidean norm $\norm{\cdot}_2$ of $H$, where $x \in H$ is identified with $(x,0)$.
Using ($\checkmark$), for each $\phi \in S$ we can find $e_\phi \in X$ such that the map $\map {i_\phi}{H \oplus \K} X$, defined by $i_\phi(x,\lambda) = x + \lambda e_\phi$, is an isometric embedding with respect to $\norm{\cdot}_\phi$.

Fix $\phi, \psi \in S$ such that $\phi \ne \psi$ and let $\norm{\cdot}$ denote the $p$-norm of $X$.
Fix $\mu > 0$ and let $w = \mu \phi \in H \subset X$.
Then
$$\norm{e_\phi - e_\psi}^p \geq \norm{e_\phi + w}^p - \norm{e_\psi + w}^p = \norm{(\mu \phi,1)}^p_\phi - \norm{(\mu \phi,1)}^p_\psi.$$
Finally, observe that $\norm{(\mu \phi,1)}^p_\phi = 1 + \mu^p$ and
$$\norm{(\mu \phi, 1)}^p_\psi = \max \Bigl\{ \mu^p, 1 + \mu^p |(\phi | \psi)|^p \Bigr\} = \mu^p,$$
whenever $\mu$ is sufficiently large, because $|(\phi | \psi)| < 1$ (recall that $\phi$, $\psi$ are distinct vectors of $S$).
Thus, we conclude that $\norm{e_\phi - e_\psi} \geq 1$ whenever $\phi \ne \psi$, which shows that the density of $X$ is at least $|S|=2^{\aleph_0}$.
\end{proof}

A couple of additional remarks about Theorem~\ref{w1} are in order. First, it is clear that any $p$-Banach space of universal disposition for the class of all separable $p$-Banach spaces must contain an isometric copy of every $p$-Banach space of density  $\aleph_1$. This is so because every quasi-Banach space $X$ of density $\aleph_1$ can be written as $X=\bigcup_{\alpha\in\omega_1}X_\alpha$, where each $X_\alpha$ is a separable subspace of $X$ and $X_\alpha\subset X_\beta$ whenever  $\alpha\leq\beta$ are countable. For the same reason, if we assume the continuum hypothesis, then we can easily obtain uniqueness up to isometries in Theorem~\ref{w1}. See Section~\ref{universal} for more on this.

\section{Some forms of injectivity for $p$-Banach spaces}\label{sec:inj}

In this Section we study the extension of operators with values in $\mathbb G_p$ and its nonseparable relatives.

\begin{definition}\label{DefVrsbtix}
Let $E$ be a $p$-Banach space.
\begin{itemize}
\item[(a)] We say that $E$ is injective amongst $p$-Banach spaces if for every $p$-Banach space $X$  and every subspace $Y$ of $X$, every operator $t:Y\to E$ can be extended to an operator $T:X\to E$. If this can be achieved with $\|T\|\leq\lambda\|t\|$ for some fixed $\lambda\geq 1$, then $E$ is said to be $\lambda$-injective amongst $p$-Banach spaces.
\item[(b)] $E$ is said to separably injective   or separably $\lambda$-injective amongst $p$-Banach spaces if the preceding condition holds when $X$ is separable.
\item[(c)] $E$ is said to be locally injective amongst $p$-Banach spaces if there is a constant $\lambda$ such that every finite-dimensional $p$-Banach space $X$  and every subspace $Y$ of $X$, every operator $t:Y\to E$ can be extended to an operator $T:X\to E$ with $\|T\|\leq\lambda\|t\|$.
\item[(d)] Finally, $E$ is called locally $1^+$-injective amongst $p$-Banach spaces if it satisfies the preceding condition for every $\lambda > 1$.
\end{itemize}
\end{definition}

These notions play a fundamental role in Banach space theory.
As it is well-known, a Banach space is injective (amongst Banach spaces) if and only if it is a complemented subspace of $\ell_\infty(I)$ for some set $I$.
Also, a Banach space is locally injective if and only if it is a $\mathscr L_\infty$-space and it is locally $1^+$-injective if and only if it is a Lindenstrauss space.

As for separable injectivity, Sobczyk theorem asserts that $c_0$ is separably 2-injective and a deep result by Zippin states that every separable separably injective Banach space has to be isomorphic to $c_0$. Nevertheless, there is a wide variety of (nonseparable) separably injective Banach spaces, see \cite{z, adv}.

\begin{proposition}\label{inj} Let $0<p<1$.
\begin{itemize}
\item[(a)]  No nonzero $p$-Banach space is injective amongst $p$-Banach spaces.
\item[(b)]  Every space of almost universal disposition for finite-dimensional $p$-Banach spaces, in particular $\mathbb G_p$, is locally $1^+$-injective amongst $p$-Banach spaces.
\item[(c)]  All spaces of universal disposition for separable $p$-Banach spaces, in particular those appearing in Theorem~\ref{w1}, are separably $1$-injective amongst $p$-Banach spaces.
\end{itemize}
\end{proposition}

\begin{proof} (a) Let $E$ be a $p$-Banach space with density character $\aleph$. Let $\mu$ denote Haar measure on the product of a family of $2^\aleph$ copies of $\mathbb T$, the unit circle. Then there is no nonzero operator from $L_p(\mu)$ to $E$ (recall that $p<1$). This was proved for $\aleph=\aleph_0$ by Kalton (see \cite[p. 163, at the end of Section 3]{k78}) and by Popov in general \cite[Theorem 1]{p}.

Thus, if we fix a nonzero $x\in E$ and we consider the subspace $\mathbb K$ of constant functions in $L_p(\mu)$, then the operator $\lambda\in\mathbb K\mapsto \lambda x\in E$ cannot be extended and $E$ is not injective.

(b) Assume $U$ is of almost universal disposition for finite-dimensional $p$-Banach spaces. Let $X$ be a finite-dimensional $p$-Banach space, $Y$ a subspace of $X$ and $t:Y\to U$ an operator of norm one. We will prove that, for each $\e>0$, there is an extension $T:X\to U$ with $\|T\|\leq1+\e$. Consider the push-out diagram
$$
\begin{CD}
Y@>>> X\\
@VtVV @VVt'V\\
t[Y]@>\imath >>  \PO
\end{CD}
$$
where the unlabelled arrow is plain inclusion. As $\imath$ is an isometry, for each $\e>0$, there is an $\e$-isometry $u:\PO\to U$ such that $\imath(t(y))=u(t(y))$ for all $y\in Y$. Then $u\circ t'$ is an extension of $t$ to $X$ with quasinorm at most $1+\e$.

(c) Replace ``finite-dimensional'' by ``separable'', take the closure of $t(Y)$, delete the word ``almost'', and put $\e=0$  in the proof of (b).
\end{proof}

\section{Universal operators on $p$-Gurari\u\i\ spaces}\label{sec:operas}

Finding operators on quasi-Banach spaces can be a difficult task. Actually, it can be an impossible task: Kalton and Roberts exhibited in \cite{rigid} a certain closed subspace of $L_p$ (for $0<p<1$) which is ``rigid'' -- every endomorphism is a multiple of the identity. Of course $\Gp$ cannot be so extreme since Theorem~\ref{completes} says that it has plenty of isometries.

Throughout this section we again fix $p \in (0,1]$.
Our aim is to construct a nonexpansive projection $\Pp$ on $\Gp$ whose kernel is isometric to $\Gp$ and satisfying the following condition:
\begin{enumerate}
	\item[$(\heartsuit)$] Given a nonexpansive operator $\map s X \Gp$, where $X$ is a separable $p$-Banach space, there exists an isometry $\map e X \Gp$ such that $\Pp \cmp e = s$.
\end{enumerate}
This will show, in particular, that $\Gp$ has nontrivial projections.

For the remaining part of the section we fix a locally $1^+$-injective separable $p$-Banach space $\Ha$.
Note that, by Proposition~\ref{inj}, we may take $\Ha = \Gp$.
In fact, besides obvious variants like the $c_0$-sum of $\Gp$ (see Corolary 6.6 below), we do not know essentially different examples, unless $p = 1$, where being locally $1^+$-injective is the same as being a Lindenstrauss space and a projection satisfying $(\heartsuit)$ has already been described in \cite{KuMetrics}.

In order to present the announced construction, we shall define a special category involving $\Ha$, which is actually a particular case of so-called comma categories. These ideas come from a recent work of Pech \& Pech~\cite{Pech} as well as from Kubi\'s~\cite{KuMetrics}, where an abstract theory of almost homogeneous structures has been developed.
Namely, let $\fK$ be the category whose objects are nonexpansive operators $u: U\to \Ha$ where $U$ is a finite-dimensional $p$-Banach space.
An $\fK$-morphism from $\map u U \Ha$ to $\map v V \Ha$ is an isometry $\map i U V$ satisfying $v \cmp i = u$.
In this case we write $\map i u v$.
Using the properties of push-outs, we easily obtain the following fact.

\begin{lemma}
$\fK$ has amalgamations.
Namely, given two $\fK$-morphisms $\map i z x$, $\map j z y$, there exist $\fK$-morphisms $i'$, $j'$ such that $i' \cmp i = j' \cmp j$.
\end{lemma}

The amalgamation property is visualized in the following commutative diagram, in which $W$ could be the push-out associated to the operators $i$ and $j$.
\begin{equation}\label{amalgamation}\xymatrix{
Y \ar[rr]^{j'} \ar[drrrrrr]_y & & W \ar[drrrr]^w \\
& & & & & & \Ha \\
Z \ar[rr]_i \ar[uu]^j \ar[urrrrrr]^z & & X \ar[uu]_{i'} \ar[urrrr]_x
}
\end{equation}

We also need the following strengthening of Lemma~\ref{key}.

\begin{lemma}\label{keypoperator}
Let $\map f X Y$ be an $\e$-isometry between finite-dimensional $p$-Banach spaces and let $\map r X \Ha$, $\map s Y \Ha$ be nonexpansive linear operators such that $s \cmp f$ is $\e$-close to $s$.
Let $X\oplus_f^\e Y$ be the space constructed in the proof of Lemma~\ref{key} and let $i,j$ be the canonical isometric embeddings of $X$ and $Y$, respectively.
Then the operator $\map {r\oplus s}  {X \oplus_f^\e Y} \Ha$, defined by $(r\oplus s)(x,y) = r(x) + s(y)$, is nonexpansive and has the property that $(r\oplus s) \cmp i = r$, $(r\oplus s) \cmp j = s$.
In particular, $i$ and $j$ become $\fK$-morphisms.
\end{lemma}

\begin{proof}
Fix $(x,y) \in X \oplus Y$ and assume $x = x_0 + x_2$, $y = y_1 - f(x_2)$.
Using the fact that $\|r(x_2) - s(f (x_2))\| \leq \e \|x_2\|$, we get
$$\|(r\oplus s)(x,y)\|^p = \| r(x_0)+r(x_2) + s(y_1) - s(f(x_2)) \|^p \leq \|x_0\|_X^p + \|y_1\|_Y^p + \e^p \|x_2\|_X^p.$$
And recalling that $\|(x,y)\|^p$ is the infimum of all expressions that can arise as the right-hand side of the preceding inequality, we see that $\|(r\oplus s)(x,y)\|^p \leq \|(x,y)\|^p$.
\end{proof}

Now we need to work within a countable subcategory of $\fK$ having certain ``density'' properties. To this end, let $D$ be a dense countable subset of $\Ha$ and let $\Ha_0$ denote the (dense) subspace of all finite linear combinations of elements of $D$ with rational coefficients. We define a subcategory $\fK_0$ of $\fK$ as follows:
\begin{itemize}
\item The objects of $\fK_0$ are those nonexpansive operators $f:F\to\Ha$ in which $F$ is a rational $p$-normed space (see the definition in Remark~\ref{RmRationsl}) and $f$ sends the rational vectors of $F$ into $\Ha_0$.
\item Given objects $f:F\to\Ha$ and $g:G\to\Ha$ in $\fK_0$, a $\fK_0$-morphism is a $\fK$-morphism $i:F\to G$ preserving ``rational'' vectors.
\end{itemize}

Let us collect the properties of $\fK_0$ we shall invoke later.

\begin{lemma}\label{invoke}
{\rm (a)} Given  $g:G\to \Ha$ in $\fK$ and $\e>0$ there is $f:F\to\Ha$ in $\fK_0$ and a surjective  $\e$-isometry $u:G\to F$ such that $\|g -f\circ u\|<\e$.

{\rm (b)} $\fK_0$ has amalgamations.
\end{lemma}

\begin{proof}
(a) By fixing a basis, we may assume $G$ is just $\mathbb K^n$ with some (not necessarily rational) $p$-norm. Take $h:G\to \Ha$ sending the unit basis to $D$ and such that $\|h\|\leq 1$ and $\|h-g\|<\e$. Let $|\cdot|$ be a rational $p$-norm on $\mathbb K^n$ such that $(1-\e)|\cdot|\leq \|\cdot\|_G\leq|\cdot|$.  Then set $F=(\mathbb K^n,|\cdot|)$, $u:G\to F$ the formal identity and $f=h\circ u^{-1}$ -- that is, $f$ is just $h$ viewed as a mapping from $F$ to $\Ha$.

(b) It is obvious that the $p$-sum of two rational $p$-normed spaces is again ``rational'' and that if $u:X\to Y$ is a rational map acting between rational $p$-normed spaces, then  there is a rational $p$-normed space $Z$, a rational surjection $\varpi: X\to Z$ and a linear isometry $v:Z\to Y/u[X]$ such that $\pi=v\circ\varpi$, where $\pi:Y\to Y/u[X]$ is the natural quotient map. This implies that everything in the amalgamation Diagram~(\ref{amalgamation}) can be done in $\fK_0$ if the initial data $i$ and $j$ are morphisms in $\fK_0$.
\end{proof}

We now construct a sequence
$$
\begin{CD}
u_0@>\imath_1>> u_1@>\imath_2 >> u_2@>>> \cdots
\end{CD}
$$
where each $u_n \colon U_n\to\Ha$ is an object of $\fK$, each arrow in the diagram above is a morphism in $\fK$, so that $\imath_{n+1}:U_n\to U_{n+1}$ is an isometry such that $u_n=u_{n+1}\circ \imath_{n+1}$, and the following condition is satisfied:
\begin{enumerate}
	\item[($\dagger$)] Given $n \in \Nat$, $\e>0$, an isometric embedding $\map e{U_n}V$ with $V$ finite-dimensional, and a nonexpansive operator $\map v V \Ha$ satisfying $v\cmp e = u_n$, there exist $m>n$ and an $\e$-isometric embedding $\map {e'} V {U_m}$ such that $e' \cmp e$ is $\e$-close to the linking map $\imath_{(n,m)}:U_n\to U_m$ and $u_m \cmp e'$ is $\e$-close to $v$.
\end{enumerate}
This can be proved by following the lines of \cite[Section 4.1]{KuMetrics}. Here we adopt a different approach which is similar to the construction from the proof of Lemma~\ref{contains}, but at each stage taking into account a fixed nonexpansive linear operator into $\Ha$.
Taking $\Ha = 0$, this will also provide an alternative construction of $\Gp$.

\begin{lemma}
There exists a sequence $u_0 \to u_1 \to u_2 \to \cdots$ satisfying $(\dagger)$.
\end{lemma}

\begin{proof}
We shall work in a countable subcategory $\fK_0$ of $\fK$ having the properties appearing in Lemma~\ref{invoke}.

First, we fix an enumeration $\{(f_n, k_n)\}_{n \in \omega}$ of all pairs of the form $(f,k)$, where $f$ is a morphism of $\fK_0$ and $k$ is a natural number, so that each pair occurs infinitely many times.
We construct a sequence $\{u_n\}_{n \in \omega}$ by induction, starting from the zero space with the zero operator.
Having defined $u_{n-1}$, we look at $(f_n, k_n)$.
If either $k_n \geq n$ or $\dom(f_n) \ne u_{k_n}$ then we set $u_n = u_{n-1}$.
So suppose that $k_n < n$ and $\dom(f_n) = u_{k_n}$.
Denote by $j$ the bonding arrow from $u_{k_n}$ to $u_{n-1}$.
Using the amalgamation property of $\fK_0$, we find arrows $g$, $h$ such that $h \cmp j = g \cmp f_n$.
Denote by $u_n$ the common co-domain of $h$ and $g$ and declare $h$ to be the bonding arrow from $u_{n-1}$ to $u_n$.

This finishes the description of the construction.
Condition ($\dagger$) follows from the ``density'' of $\fK_0$ in $\fK$ and from the fact that each appropriate pair $(f,k)$ appears infinitely many times in the enumeration.
\end{proof}

Actually, it can be shown that ($\dagger$) specifies the sequence $\{u_n\}$, up to an isomorphism in the appropriate category, although we shall not use this fact.

Consider the directed system of $p$-Banach spaces underlying the sequence we have just constructed:
$$
\begin{CD}
U_0@>\imath_1>> U_1@>\imath_2 >> U_2@>>> \cdots
\end{CD}
$$
set $U_\infty=\underrightarrow{\lim} \: U_n$ and let $u_\infty:U_\infty\to\Ha$ be direct limit of the operators $u_n$.
The main properties of $u_\infty$ are collected below.

\begin{theorem}\label{Thmrbgibr}
The space $U_\infty$ and the operator $u_\infty: U_\infty\to \Ha$ have the following properties:
\begin{enumerate}
	\item[(a)] The operator $u_\infty$ is nonexpansive and right-invertible -- in particular, its range is $\Ha$.
	\item[(b)] Both $U_\infty$ and $\ker(u_\infty)$ are linearly isometric to $\Gp$.
	\item[(c)] For every nonexpansive linear operator $\map s X \Ha$, where $X$ is a separable $p$-Banach space, there exists a linear isometric embedding $\map e X U_\infty$ such that
$s = u_\infty \cmp e.$
\end{enumerate}
\end{theorem}

Note that if $\Ha = \Gp$ then $\Pp = r \cmp u_\infty$, where $\map r \Gp \Gp$ is a fixed right inverse of $u_\infty$, provides the ``universal'' projection announced at the beginning of the Section.

\begin{proof}
Obviously, $u_\infty$ is nonexpansive, being a direct limit of nonexpansive operators. That it is right-invertible will follow from (c), just taking $s$ as the identity on $\Ha$.

In order to prove (c), fix a nonexpansive operator
$\map s X \Ha$ from a separable $p$-Banach space $X$ and let $(X_n)$ be an increasing sequence of finite-dimensional subspaces whose union is dense in $X$, with $X_0=0$.
Set $s_n=s\restriction X_n$ and $\e_n=2^{-n/p}$.

We shall construct inductively nonexpansive $\e_n$-isometries $\map {e_n}{X_n}{U_{k_n}}$ so that the following condition is satisfied:
\begin{enumerate}
	\item[($*$)]\quad $u_{k_n} \cmp e_n$ is $\e_n$-close to $s_n$ and $e_{n+1} \restriction X_n$ is $(\e_n^p+\e_{n+1}^p)^{1/p}$-close to $\imath_{(k_n,k_{n+1})}\circ e_n$.
\end{enumerate}
Having defined $e_n: X_n\to U_{k_n}$ with $\|s_n-u_{k_n}\circ e_n\|\leq\e_n$, we may apply Lemma~\ref{keypoperator} with $f=e_n$ to get the commutative diagram
$$\xymatrix{
U_{k_n} \ar[d]_j \ar[drr]^{u_{k_n}}  \\
X_n\oplus_{e_n}^{\e_n} U_{k_n} \ar[rr]^{s_n\oplus u_{k_n}}& & \Ha \\
 X_n \ar[u]^i \ar[urr]_{s_n}\\
}$$
which shows that $i$ is a $\fK$-morphism from $s_n$ to  ${s_n\oplus u_{k_n}}$.  On the other hand, the inclusion  of $X_n$ into $X_{n+1}$ we momentarily denote by $a$ is a $\fK$-morphism from $s_n$ to $s_{n+1}$ and amalgamating $i$ and $a$ we arrive to the following commutative diagram
 $$\xymatrix{
 U_{k_n} \ar[dd]_j  \ar[dddrrrrrr]^{u_{k_n}}\\
 &&&&&&\\
X_n\oplus_{e_n}^{\e_n} U_{k_n}  \ar[rr]^{a'} \ar[drrrrrr]_{\quad s_n\oplus u_{k_n}} & & W \ar[drrrr]^w \\
& & & & & & \Ha \\
X_n \ar[rr]_a \ar[uu]^i \ar[urrrrrr]^{s_n} & & X_{n+1} \ar[uu]^{i'} \ar[urrrr]_{s_{n+1}}
}$$
Here, $W$ is a finite-dimensional $p$-normed space and $w$ nonexpansive.
Having $(\dagger)$ in mind we can find $k_{n+1}>k_n$ and a nonexpansive $\e_{n+1}$-isometry $\ell: W\to U_{k_{n+1}}$ such that $u_{k_{n+1}}\circ\ell$ is $\e_{n+1}$-close to $w$ and $\ell\circ a'\circ j$ is $\e_{n+1}$-close to $\imath_{(k_n, k_{n+1})}:U_{k_n}\to U_{k_{n+1}}$.

Setting $e_{n+1}=\ell\circ i': X_{n+1}\to W\to U_{k_{n+1}}$ we obtain an $\e_{n+1}$-isometry fulfilling the requirements in $(*)$. Indeed,
$$
\|u_{k_{n+1}}\circ e_{n+1}-s_{n+1}\|= \|u_{k_{n+1}}\circ \ell\circ i' -s_{n+1}\|= \|u_{k_{n+1}}\circ \ell\circ i' - w\circ i'\|\leq \|u_{k_{n+1}}\circ \ell - w\|\leq \e_{n+1},
$$
while
\begin{align*}
\|\imath_{(k_n, k_{n+1})}\circ e_n -e_{n+1}\circ a\|^p&\leq \|\imath_{(k_n, k_{n+1})}\circ e_n -\ell\circ a'\circ j\circ e_n\|^p +
 \|\ell\circ a'\circ j\circ e_n -e_{n+1}\circ a\|^p\\
 & = \|\imath_{(k_n, k_{n+1})}\circ e_n -\ell\circ a'\circ j\circ e_n\|^p +
 \|\ell\circ a'\circ j\circ e_n -\ell\circ i' \circ a\|^p\\
 &= \|\imath_{(k_n, k_{n+1})}\circ e_n -\ell\circ a'\circ j\circ e_n\|^p +
 \|\ell\circ a'\circ j\circ e_n -\ell  \circ a'\circ i\|^p\\
 &\leq \|\imath_{(k_n, k_{n+1})}-\ell\circ a'\circ j\|^p \cdot \|e_n\|^p +
 \|\ell\circ a'\|^p \cdot \|j\circ e_n-i\|^p\\
 &\leq \e_{n+1}^p+\e_n^p
 \end{align*}
since $\|j\circ e_n-i\|\leq\e_n$.
Finally, the sequence $\{e_n\}$ ``converges'' to an isometric embedding $e:X\to U_\infty$ satisfying $u_\infty \cmp e = s$, which proves (c).

We pass to the proof of (b). To prove that $U_\infty$ is isometric to $\mathbb G_p$, it suffices to prove that it is of almost universal disposition for the class of finite-dimensional $p$-Banach spaces and we shall show that $U_\infty$ satisfies the hypothesis of Lemma~\ref{relax}. Here, the hypothesis that $\Ha$ is locally $1^+$-injective amongst $p$-Banach spaces enters.

So, let $g: X\to Y$ be an isometry, where $X$ is a subspace of $U_\infty$ and $Y$ a finite-dimensional $p$-normed space. As we did in the proof of Lemma~\ref{contains}, we fix a small $\delta>0$ and we take a nonexpansive $\delta$-isometry $w:X\to U_n$ such that $\|w(x)-x\|_{U_\infty}\leq \delta\|x\|$. Let us form the push-out square
$$
\begin{CD}
X@> g >> Y\\
@V w VV @VV w'V\\
U_n @>g'>> \PO
\end{CD}
$$
Here $g'$ is an isometry and $w'$ is a contractive $\delta$-isometry, according to Lemma~\ref{po}.
As $\Ha$ is locally $1^+$-injective and $\PO$ is finite-dimensional there is an operator $\widehat u_n: \PO\to\Ha$ such that $u_n=\widehat u_n\circ g'$, with $\|\widehat u_n\|\leq 1+\delta$. Next we amend the $p$-norm in $\PO$ to make $\widehat u_n$ nonexpansive: for instance we may take $|v|=\max(\|v\|_{\PO}, \|\widehat u_n(v)\|_\Ha)$. Now, if $V$ denotes the space $\PO$ furnished with this new norm, then $g':U_n\to V$ is still isometric and $\widehat u_n: V\to\Ha$ becomes contractive and we may use $(\dagger)$ to get $m>n$ and a $\delta$-isometric embedding $g'':V\to U_m$ such that $\|g''\circ g'-\imath_{(n,m)}\|\leq \delta$. Now, if $\delta>0$ is small enough, the composition
$$
\begin{CD}
f: Y@> w' >> \PO @>{\bf {\rm Id}}>> V @> g'' >> U_m @>>> U_\infty
\end{CD}
$$
is an $\e$-isometry such that $\|f(g(x))-x\|_U\leq \e\|x\|$ for every $x\in X$. This shows that $U_\infty$ is isometric to $\mathbb G_p$.
\medskip

It only remains to check that $\ker u_\infty$ is isometric to $\Gp$. We first prove that the operator $u_\infty:U_\infty\to\Ha$ has the following additional property:
\begin{itemize}
	\item[($\ddagger$)] Suppose $E$ is a subspace of a finite-dimensional $p$-Banach space $F$. If $g:F\to \Ha$ is nonexpansive and $e:E\to U_\infty$ is an isometry such that $u_\infty\circ e=g\restriction E$, then for each $\delta>0$ there is a $\delta$-isometry $f:F\to U_\infty$ satisfying
$\| f \restriction E - e \| < \delta$ and  $ \| u_\infty \cmp f - g \| < \delta$.
\end{itemize}

Indeed, after taking a small perturbation, we may assume that $\map e E {U_n}$ is an $\e$-isometric embedding and $u_\infty \cmp e$ is $\e$-close to $g \restriction E$. Applying Lemma~\ref{keypoperator} to $e:E\to U_n$ and the operators $g:E\to\Ha$ and $u_n:U_n\to\Ha$, we get the commutative diagram
$$\xymatrix{
U_{n} \ar[d]_j \ar[drr]^{u_{n}}  \\
E\oplus_e^\e U_{n} \ar[rr]^{g\oplus u_{n}}& & \Ha \\
 E \ar[u]^i \ar[urr]_{g}\\
}$$
 with $\|j \circ e- i\|\leq\e$. Now, amalgamating $i:E\to E\oplus_e^\e U_{n}$ which is a morphism from $g\restriction E$ to $g\oplus u_n: E\oplus_e^\e U_n\to\Ha$ and the inclusion of $E\to F$ regarded as a morphism from $g\restriction E$ to $g:F\to\Ha$ we obtain a finite-dimensional $p$-normed space $W$ and a commutative diagram
 $$\xymatrix{
 U_{n} \ar[dd]_j  \ar[dddrrrrrr]^{u_{n}}\\
 &&&&&&\\
E\oplus_e^\e U_{n}  \ar[rr]^{a'} \ar[drrrrrr]_{s_n\oplus u_{k_n}} & & W \ar[drrrr]^w \\
& & & & & & \Ha \\
E\ar[rr]_a \ar[uu]^i \ar[urrrrrr]^{s_n} & & F \ar[uu]^{i'} \ar[urrrr]_{g}
}$$
with $\|w\|\leq 1$. Applying now $(\dagger)$ to $w$ and the embedding $a'\circ j$ we obtain $m>n$ and an almost isometry $v:W\to U_m$ such that $u_m\circ v$ is close to $w$ and  $v\circ a'\circ j $ is close to $\imath_{(n,m)}$. Finally, the composition
$$
\begin{CD}
f: F@> i' >>  W @> v >> U_m @>>> U_\infty
\end{CD}
$$
does the trick.

After this preparation, let $F$ be a finite-dimensional $p$-normed space, $e:E\to\ker u_\infty$ an isometry, where $E$ is a subspace of $F$, and $\e>0$. We shall construct an $\e$-isometry $f:F\to\ker u_\infty$ such that $\|f(x)-e(x)\|\leq \e\|x\|$ for every $x\in E$. This will show that $\ker u_\infty$ is of almost universal disposition, thus completing the proof.

Fix some small $\delta$ and apply $(\ddagger)$ taking $g$ as the zero operator from $F$ to $\Ha$ to get a $\delta$-isometry $f':F\to U_\infty$ such that
$\| f' \restriction E - e \| < \delta$ and  $ \| u_\infty \cmp f' \| < \delta$. Of course, we cannot guarantee that $f'$ takes values in $\ker u_\infty$. To amend this, let $r:\Ha\to U_\infty$ be a right-inverse for $u_\infty$, and set $f=({\bf 1}_{U_\infty}-r\circ u_\infty)\circ f'$, that is, $f(x)=f'(x)-r(u_\infty(f'(x)))$. Then $f$ takes values in $\ker u_\infty$ since $u_\infty\circ f=0$ and, moreover, $\|f-f'\|=\|r\circ u_\infty\circ f'\|\leq \delta$. Thus for $\delta$ sufficiently small $f:F\to\ker u_\infty$ is an $\e$-isometry with $\| f \restriction E - e \| < \e$.
\end{proof}

\begin{corollary}
The following spaces are linearly isomorphic to $\Gp$:
\begin{itemize}
\item[(a)] $\Gp\oplus\Gp$ as well as all finite direct sums $\Gp \oplus \dots \oplus \Gp$.
\item[(b)] $c_0(\Gp)$, the space of all sequences converging to $0$ in $\Gp$, endowed with the maximum $p$-norm.
\item[(c)] $C(\Delta, \Gp)$, the space of all continuous functions from the Cantor set $\Delta$ to $\Gp$, with the sup quasinorm.
\end{itemize}
\end{corollary}

\begin{proof}
We first observe that if $\mathbb H$ is separable and $1^+$-locally injective amongst $p$-Banach spaces, then $\Gp\oplus\mathbb H$ is linearly isomorphic to $\Gp$. Indeed, if $u_\infty: U_\infty\to \mathbb H$ is the operator appearing in Theorem~\ref{Thmrbgibr} and $r:\mathbb H\to U_\infty$ is a right inverse for $u_\infty$, then the mapping $x\in U_\infty\mapsto (x-r(u_\infty(x)),u_\infty(x))\in \ker(u_\infty)\oplus\mathbb H$ is a linear homeomorphism and we already know that both $U_\infty$ and $\ker(u_\infty)$ are isomorphic to $\Gp$.

Now, taking $\mathbb H=\Gp$ which is locally $1^+$-injective according to Proposition~\ref{inj}(b), we see that $\Gp$ is isomorphic to $\Gp \oplus \Gp$ hence also to any finite sum $\Gp \oplus \dots \oplus \Gp$. This proves (a).

To prove (b) and (c), note that both $c_0(\Gp)$  and $C(\Delta, \Gp)$ are  locally $1^+$-injective, being the completion of the union of a chain of spaces of the form $\Gp \oplus \dots \oplus \Gp$ endowed with the maximum quasinorm; all these spaces are locally $1^+$-injective, because $\Gp$ is.
Hence, $\Gp$ is isomorphic to $\Gp \oplus c_0(\Gp)$ which is isomorphic to $c_0(\Gp)$, which proves (b).
As for (c), we know that $\Gp$ is isomorphic to $\Gp\oplus C(\Delta, \Gp)$ and since $\Gp$ lives complemented in $C(\Delta, \Gp)$ (as the subspace of constant functions) and $\Gp$ is isomorphic to its square, Pe\l czy\'nski decomposition method applies: indeed, if $C(\Delta,\Gp)\approx \Gp \oplus Z$, then $\Gp\approx \Gp \oplus C(\Delta, \Gp)\approx \Gp \oplus \Gp \oplus Z \approx  \Gp \oplus Z\approx C(\Delta,\Gp)$.
\end{proof}


\section{Miscellaneous remarks and questions}\label{closing}

\subsection{Mazur's ``rotations'' problem}
A quasi-Banach space is said to be almost isotropic if the orbits of the isometry group are dense in the unit sphere: if $\|x\|=\|y\|= 1$, then for every $\e>0$ there is a surjective isometry $u$ such that $\|y-u(x)\|\leq \e$. If this condition holds even for $\e=0$, the space is said to be isotropic: the isometry group acts transitively on the sphere.

A notorious problem that Banach attributes to Mazur in his {\it ``Th\'eorie des Op\'erations Lin\'eaires''} asks whether $\ell_2$ is the only separable isotropic Banach space; cf. \cite[p. 242]{banach}. This is the problem mentioned by Gurari\u\i\ in the title of \cite{g} and, as far as we know, is still open. We may refer the reader to \cite{c-extracta, b-rp} for two complementary surveys on the topic.

The following remark is immediate from Theorem~\ref{completes}.

\begin{corollary}\label{almost isotropic}
The space $\mathbb G_p$ is almost isotropic. \hfill$\square$
\end{corollary}

It is well-known that $\mathbb G$ (``our'' $\mathbb G_p$ when $p=1$) is not isotropic. However the standard argument depends on Mazur's theorem about the existence of smooth points on any separable Banach space and this argument is not available when $p<1$.

It is worth remarking that the notion of ``almost isotropic space'' that Gurari\u\i\ manages in \cite{g} is  stronger than ours: For every $\e > 0$, every linear isomorphism $f$ between finite-dimensional subspaces should extend to a bijective linear isomorphism $\tilde f$ satisfying $\| \tilde f \| \leq (1+\e) \|f\|$ and $\| \tilde f ^{-1} \| \leq (1+\e) \|f\|$.
Anyway, it is clear from the proof of \cite[Theorem 3]{g} that the spaces $\Gp$ are ``almost isotropic'' in Gurari\u\i\/ 's sense for all $p\in(0,1]$.

\subsection{Ultrapowers of $\mathbb G_p$}
There is an alternative proof of Theorem~\ref{w1} which is based on the ultraproduct construction; see \cite{k84}. Let $(X_i)$ be a family of $p$-Banach spaces indexed by $I$ and let $\mathscr U$ be a countably incomplete ultrafilter on $I$. Then the space of bounded families $\ell_\infty(I,X_i)$ with the quasinorm $\|(x_i)\|=\sup_i \|x_i\|$ is a $p$-Banach space and $c_0^\mathscr U(X_i)=\{(x_i): \lim_\mathscr U \|x_i\|=0\}$ is a closed subspace of $\ell_\infty(I,X_i)$. The ultraproduct of the family $(X_i)$ with respect to $\mathscr U$, denoted by $[X_i]_\mathscr U$, is the quotient space $\ell_\infty(I,X_i)/c_0^\mathscr U(X_i)$ with the quotient quasinorm. The class of the family $(x_i)$ in $(X_i)_\mathscr U$ is denoted by $[(x_i)]$.

 The quasinorm in $[X_i]_\mathscr U$ can be computed as $\|[(x_i)]\|=\lim_\mathscr U\|x_i\|$.

When all the spaces $X_i$ are the same, say $X$, the ultraproduct is called the ultrapower of $X$ following $\mathscr U$.
One has the following generalization of \cite[Proposition 5.7]{accgm} for which we provide a simpler proof.

\begin{proposition}
If $\mathscr U$ is a countably incomplete ultrafilter on the integers, then $[\mathbb G_p]_\mathscr U$ is a $p$-Banach space of universal disposition for separable $p$-Banach spaces and its density character is the continuum.
\end{proposition}

\begin{proof}
We denote by $I$ the index set supporting $\mathscr U$. Let $X$ be a separable subspace of $[\mathbb G_p]_\mathscr U$ and $g:X\to Y$ an isometry, where $Y$ is any separable $p$-Banach space. We will prove that there is an isometry $f:Y\to [\mathbb G_p]_\mathscr U$ such that $f(g(x))=x$ for every $x\in X$.
Clearly, we may and do assume that $Y/X$ has dimension one.

So, let $(x^n)$ be a normalized, linearly independent sequence whose linear span is dense in $X$ and $y^0\in Y\backslash X$.
Let $X^n$ be the subspace spanned by $(x^1,\dots, x^n)$ in $X^n$ and $Y^n$ the subspace spanned by $g[X^n]$ and $y^0$ in $Y$.

Also,
 let us fix representatives $(x_i^n)$ so that $x^n=[(x_i^n)]$ for every $n$. We may assume $\|x_i^n\|=1$ for every $n$ and every $i$.
For $i\in I$ and $n\in\mathbb N$, let us denote by $X_i^n$ the subspace of $\mathbb G_p$ spanned by $(x_i^1,\dots,x_i^n)$. We define a linear map $\jmath_{n,i}: X^n\to X^n_i$ by letting $\jmath_{n,i}(x^k)= x^k_i$ for $1\leq k\leq n$ and linearly on the rest.

To proceed, we observe that the sets
$$
I^n_\e=\{i\in I \text{ such that $\jmath_{n,i}: X^n\to X^n_i$ is a strict $\e$-isometry}\}
$$
are in $\mathscr U$ for every $n$ and every $\e>0$. Let $(J_n)$ be a sequence of subsets of $\mathscr U$ with $\bigcap_n J_n=\varnothing$. For each $i\in I$, set $n(i)=\max\{n\in\mathbb N: i\in J_n\cap I^n_{1/n}\}$ and observe that $n(i)\to\infty$ along $\mathscr U$.

Let us form the ultraproducts $[X^{n(i)}]_\mathscr U, [X^{n(i)}_i]_\mathscr U$ and $[Y^{n(i)}_i]_\mathscr U$.
It is  obvious that $[X^{n(i)}]_\mathscr U$ and $[X^{n(i)}_i]_\mathscr U$ are isometric through the ultraproduct operator $[(\jmath_{n(i),i})]$.
Moreover, there is a linear isometry $\kappa: X\to [X^{n(i)}]_\mathscr U$ that we may define taking $\kappa(x)=[(x_i)]$, where $x_i\in X^{n(i)}$ is any point minimizing the ``distance'' from $x$ to $X^{n(i)}$ and the same applies to $Y$ and $[Y^{n(i)}_i]_\mathscr U$.

For each $i\in I$ consider the composition $g\circ\jmath_{n(i),i}^{-1}$ which is a strict $(1/n(i))$-isometry from $X^{n(i)}_i$ into $Y^{n(i)}$. On account of Lemma~\ref{helpful} we may find an $(1/n(i))$-isometry $f_i: Y^{n(i)}\to \mathbb G_p$ such that $\|f_i(g(\jmath_{n(i),i}^{-1}(x)))-x\|\leq \|x\|/n(i)$ for every $x\in X^{n(i)}_i$. It is now obvious that if $f:Y\to [\mathbb G_p]_\mathscr U$ denotes the composition  of the embedding $Y\to [Y^{n(i)}_i]_\mathscr U$ with the ultraproduct operator $[(f_i)]$ one obtains an isometry with $f(g(x))=x$ for every $x\in X$.
\end{proof}

Notice that, while it is unclear whether the spaces arising in the proof of Theorem~\ref{w1} are isotropic or not, it follows from Corollary~\ref{almost isotropic} and rather standard ultraproduct techniques that every ultrapower of $\mathbb G_p$ built over a countably incomplete ultrafilter is isotropic.

\subsection{Universal spaces}\label{universal}
As we have already mentioned, under the continuum hypothesis, all the spaces having the properties appearing in Theorem~\ref{w1} are isometric. It was observed in \cite[Proposition 4.7]{accgm} that, in the Banach space setting, the uniqueness cannot be proved in ${\sf ZFC}$, the usual setting of set theory, with the axiom of choice. This depends on the fact that it is consistent with ${\sf ZFC}$ that there is no Banach space of density $2^{\aleph_0}$ containing an isometric copy of all Banach spaces of density $2^{\aleph_0}$, a recent result by Brech and Koszmider \cite{b-k}. Whether or not  the same happens to $p$-Banach spaces is left open to reflection.

\subsection{Vector-valued Sobczyk's theorem without local convexity}
Sobczyk's theorems states that $c_0$, the Banach space of all sequences converging to zero with the sup norm is separably injective -- amongst Banach spaces, of course. More interesting for us is that if $E$ is a separably injective Banach space, then so is $c_0(E)$ -- the space of sequences converging to $0$ in $E$. Several proofs of this fact are available. Some of them made strong use of local convexity. For instance, Johnson-Oikhberg's argument in \cite{j-o} is based on $M$-ideal theory, while Castillo-Moreno proof in \cite{c-m} uses the bounded approximation property, a very rare property outside the Banach space setting. We don't know if Rosenthal's proof
in \cite{ros} would survive without local convexity or not, but in any case the proof in \cite{c} applies verbatim to $p$-Banach spaces. So we have the following.

\begin{proposition}
If $E$ is separably injective amongst $p$-Banach spaces, then so is $c_0(E)$.
\end{proposition}

We do not know whether there is a nontrivial separable space, separably injective amongst $p$-Banach spaces when $p<1$, but our guess is no.
In any case, such a space would necessarily be a complemented subspace of $\mathbb G_p$.

\subsection{Operators on $\mathbb G_p$ when $p<1$} It is a classical result in quasi-Banach space theory that every operator from $L_p$ to a $q$-Banach space for $p<q\leq 1$ is zero. It follows easily that the same is true replacing $L_p$ by $\Gp$. In particular, the dual of $\Gp$ is trivial.
In a similar vein, there is no nonzero operator from $\mathbb G_p$ into any $L_q$ (here $q$ can be 0) and there is no compact operator on $\mathbb G_p$; the first statement follows from the fact that there is no nonzero operator from $L_p/H_p$ to $L_0$, see \cite{alek} and the second one from the fact that every operator defined on $L_p$ is either zero or an isomorphism on a copy of $\ell_2$, see \cite[Theorem 7.20]{kpr} for which is perhaps the simplest proof.

We do not know whether $\Gp$ is isomorphic to all its quotients or complemented subspaces.
In particular we don't know whether $\Gp$ is isomorphic to its quotient by a line.

This is clearly connected to the notion of a $K$-space. Recall that a quasi-Banach space $X$ is said to be a $K$-space if whenever $Z$ is a quasi-Banach space with a subspace $L$ of dimension one such that $Z/L$ is isomorphic to $X$, then $L$ is complemented in $Z$ and so $Z$ is isomorphic to $\mathbb K\oplus X$.

It would be interesting to know whether the spaces $\Gp$ are $K$-spaces or not. The case $p=1$ is solved in the affirmative by a deep result of Kalton and Roberts \cite[Theorem 6.3]{kr}, who proved that every $\mathscr L_\infty$-space, and in particular the Gurari\u\i\ space, is a $K$-space.

\subsection*{Acknowledgments}
The proof of Proposition~\ref{PBenYakHay} is due to Haydon who presented it to the third named author during the {\it Workshop on Forcing Axioms and their Applications} held at the Fields Institute, Toronto, 22 -- 26 October 2012.
We thank Richard Haydon for this elegant argument.


\begin{thebibliography}{00}


\bibitem{alek}
AB Aleksandrov,
Essays on nonlocally convex Hardy classes. Complex analysis and spectral theory (Leningrad, 1979/1980), pp. 1--89,
Lecture Notes in Math., 864, Springer, Berlin-New York, 1981


\bibitem{accgm}
A Avil\'es {\it et al.}, Banach spaces of universal disposition, J. Funct. Anal. 261 (2011) 2347--2361


\bibitem{adv}
A Avil\'es {\it et al.}, On separably injective Banach spaces. Adv. Math. 234 (2013) 192--216

\bibitem{banach}
S Banach, Th\'eorie des Op\'erations Lin\'eaires, Monografie Matematyczne
1, Warsawa-Lw\'ow, 1932.



\bibitem{b-rp}
J Becerra, \'A Rodr\'\i guez Palacios, Transitivity of the norm on Banach spaces, Extracta Math. 17 (2002) 1--58


\bibitem{BYH}
I Ben Yaacov, CW Henson, {Generic orbits and type isolation in the Gurarij space}, preprint, \href{http://arxiv.org/abs/1211.4814}{arXiv:1211.4814v1}


\bibitem{b-l}
Y Benyamini, J Lindenstrauss,
A predual of $l_1$ which is not isomorphic to a $C(K)$ space,
Israel J. Math. 13 (1972) 246--254


\bibitem{b-k}
C Brech, P Koszmider, On universal Banach spaces of density continuum, Israel J. Math. 190 (2012), 93--110


\bibitem{c-extracta}
F Cabello S\'anchez, Regards sur le probl\`{e}me des rotations de Mazur, Extracta Math. 12 (1997), no. 2, 97--116


\bibitem{c}
F Cabello S\'anchez,
Yet another proof of Sobczyk's theorem. In: Methods in Banach Space Theory. Edited by JMF Castillo and WB Johnson.
London Math Soc Lecture Note Series No. 337, Cambridge 2006


\bibitem{c-m}
JMF Castillo, Y Moreno, Sobczyk's theorem and the bounded approximation property. Studia Math. 201 (2010), no. 1, 1--19


\bibitem{cie}
K Ciesielski, Set theory for the working mathematician, London Math. Soc. Student Texts 39. CUP, 1997


\bibitem{g-k}
J Garbuli\'nska, W Kubi\'s, Remarks on Gurari\u\i\ spaces, Extracta Math. 26 (2011) 235--269


\bibitem{g}
VI Gurari\u\i,
Spaces of universal placement, isotropic spaces and a problem of Mazur on rotations of Banach spaces, Sibirsk. Mat. Zh. 7 (1966) 1002--1013 (in Russian)


\bibitem{j-o}
WB Johnson, T Oikhberg, Separable lifting property and extensions of local reflexivity, Illinois J. Math. 45 (2001) 123--137


\bibitem{k77}
NJ Kalton, Universal spaces and universal bases in metric linear spaces, Studia Math. 61 (1977) 161--191


\bibitem{k78}
NJ Kalton, Transitivity and quotients of Orlicz spaces, Comment. Math. (Special issue in honor of the 75th birthday of W. Orlicz) (1978) 159--172


\bibitem{k84}
NJ Kalton, Locally complemented subspaces and $\mathscr L_p$ spaces  for $p<1$, Math. Nachr. 115 (1984) 71--97


\bibitem{kpr}
NJ Kalton, NT Peck, JW Roberts, An F-space sampler, London Math. Soc. Lecture Note Series no. 89. Cambridge 1984

\bibitem{rigid}
NJ Kalton, JW Roberts,
A rigid subspace of $L_0$, Trans. Amer. Math. Soc. 266 (1981), 645--654


\bibitem{kr}
NJ Kalton, JW Roberts,
Uniformly exhaustive submeasures and nearly additive set functions,
Trans. Amer. Math. Soc. 278 (1983) 803--816


\bibitem{KuMetrics}
W Kubi\'s, Metric-enriched categories and approximate \fra\ limits,
preprint, \href{http://arxiv.org/abs/1210.6506}{arXiv:1210.6506}


\bibitem{ks}
W Kubi\'s, S Solecki, A proof of uniqueness of the Gurarii space, Israel J. Math. 195 (2013), no. 1, 449--456




\bibitem{l-unique}
W Lusky, The Gurari\u\i\ spaces are unique, Arch. Math. 27 (1976) 627--635


\bibitem{l-survey}
W Lusky, Separable Lindenstrauss spaces. Functional Analysis: Surveys and Recent Results. North-Holland (1977) pp. 15--28

\bibitem{group}
P Niemiec,
Universal valued Abelian groups, Adv. Math. 235 (2013) 398--449




\bibitem{Pech}
C Pech, M Pech, Universal homomorphisms, universal structures, and the polymorphism clones of homogeneous structures, preprint,
arXiv: \href{http://arxiv.org/abs/1302.5692}{1302.5692}


\bibitem{p}
MM Popov, Codimension of subspaces of $L_p$ for $0<p<1$,
Funktional. Anal. i Prilozhen. 18 (1984), no. 2, 94--95


\bibitem{r}
S Rolewicz,
Metric linear spaces.
Second edition. Mathematics and its Applications (East European Series), 20. D Reidel Publishing Co, Dordrecht; PWN--Polish Scientific Publishers, Warsaw, 1985


\bibitem{ros}
HP Rosenthal, The complete separable extension problem, J. Operator Theory 43 (2000) 329--374


\bibitem{Wojt}
P Wojtaszczyk,
{Some remarks on the Gurarij space},
Studia Math. {41} (1972) 207--210


\bibitem{z}
M Zippin, Extension of bounded linear operators, in: WB Johnson, J Lindenstrauss (Eds), Handbook of the
Geometry of Banach Spaces, vol. 2, Elsevier (2003) pp. 1703--1741

\end{thebibliography}
\end{document}